\documentclass[reqno]{amsart}
\usepackage{amsmath,amssymb}
\usepackage{graphics,epsfig,color}
\usepackage[mathscr]{euscript}

\usepackage{graphicx}
\usepackage{graphics}
\usepackage{caption}
\usepackage{subfig}
\usepackage{hyperref}
\usepackage{tikz}
\usetikzlibrary{arrows}

\usepackage{bm}
\theoremstyle{plain}
\newtheorem{theorem}{Theorem}[section]

\newtheorem{lemma}[theorem]{Lemma}

\theoremstyle{definition}
\newtheorem{definition}[theorem]{Definition}

\theoremstyle{remark}
\newtheorem{remark}[theorem]{Remark}

\numberwithin{equation}{section}
\numberwithin{theorem}{section}
\newcommand{\upbar}[1]{\,\overline{\! #1}}

\renewcommand{\epsilon}{\varepsilon}
\renewcommand{\tilde}{\widetilde}

\renewcommand{\div}{\mathop{\rm div}\nolimits}

\definecolor{light}{gray}{.9}

\title[Invariant measures on metric graphs]{A combinatorial representation for the
invariant measure of diffusion processes on metric graphs}

\author{Michele Aleandri}
\address{\small{Universit{\`a} LUISS Guido Carli \\ Viale Romania, 32 \\ 00197 Roma, Italia. }}
\email{maleandri@luiss.it}

\author{Matteo Colangeli}
\address{\small{Universit{\`a} dell'Aquila\\ Via Vetoio, Loc. Coppito\\ 67010 L'Aquila, Italia.}}
\email{matteo.colangeli1@univaq.it}

\author{Davide Gabrielli}
\address{\small{Universit{\`a} dell'Aquila\\ Via Vetoio, Loc. Coppito\\ 67010 L'Aquila, Italia.}}
\email{davide.gabrielli@univaq.it}

\begin{document}

\begin{abstract}

We give a generalization to a continuous setting of the classic Markov chain tree Theorem. In particular, we consider an irreducible  diffusion process on a metric graph. The unique invariant measure has an atomic component on the vertices and an absolutely continuous part on the edges. We show that the corresponding density at $x$ can be represented by a normalized superposition of the weights associated to metric  arborescences oriented toward the point $x$. The weight of each oriented metric arborescence is obtained by the exponential of integrals of the form $\int\frac{b}{\sigma^2}$ along the oriented edges time a weight for each node determined by the local orientation of the arborescence around the node time the inverse of the diffusion coefficient at $x$. The metric arborescences are obtained cutting the original metric graph along some edges.
\bigskip


\end{abstract}

\maketitle
\thispagestyle{empty}

\section{Introduction}
A powerful construction for finite state Markov chain is the so-called \emph{Markov Chain Matrix Tree Theorem} \cite{AT, FWb, PW}. In the case of an irreducible continuous time finite state Markov chain the unique invariant measure is obtained as a normalized superposition of weights associated to some combinatorial structures. The combinatorial structures considered are the rooted arborescences of the transition graph. The transition graph of the chain is a directed graph with vertices corresponding to the states and directed edges corresponding to the possible transitions. For an irreducible chain the graph is strongly connected. An arborescence is a spanning directed subgraph such that, disregarding the orientation, we have a spanning tree, and moreover all the edges are directed towards a single vertex $x$, called the root. The weight of each arboresence is the product of the rates of all the edges that it contains.
The invariant measure at $x$ coincides with the normalized sum of the weights of all arborescences rooted at $x$.

\smallskip
A classic result in Probability Theory is the diffusive rescaling of a class of random walks with convergence to diffusion processes. The prototype of this class of results is the celebrated \emph{Donsker Theorem}. A diffusive random walk is obtained, generically, by weakly perturbing  a reversible random walk. The most general reversible random walk on a graph is determined by some positive weights associated to vertices and symmetric positive weights associated to the edges. In the scaling limit, we consider a grid
of mesh $1/N$ embedded into $\mathbb R^d$ and consider the weights on the vertices and edges as discretized versions of smooth positive functions that
we call $\alpha$ and $Q$ respectively. Likewise, the weak perturbation is obtained by the discretization of a smooth vector field $F$. The family of diffusive walks is therefore parameterized by the triple $(\alpha, Q, F)$. Correspondingly the family  of limiting processes, that are the diffusion processes, are parameterized by the pair $(b.\sigma)$ (see equation \eqref{diff}) where the vector field $b$ is called the \emph{drift} and the matrix $\sigma$ is called the \emph{diffusion coefficient}. The correspondence between $(\alpha, Q, F)$ and $(b,\sigma)$ is not
one-to-one and a whole class of microscopic models converge to the same diffusion process.

\smallskip

A very natural question is whether the combinatorial representation of the invariant measure for the random walks has a corresponding continuous version for the diffusion processes. We answer positively to this question in the case of diffusions on \emph{metric graphs}. A metric graph is a metric space that is obtained by gluing the extrema of a finite number of bounded segments to some vertices. A diffusion process on a metric graph is a process that evolves like a diffusion along each edge and then, when reaching a vertex, it evolves by possibly spending some random time therein and by then picking up at random a new edge on which the evolution continues.

\smallskip
We prove in this paper that the invariant measure of a diffusion process on a metric graph has a representation that is the continuous counterpart of the combinatorial representation of the matrix tree Theorem. More precisely we have that the density of the invariant measure on a point $x$ belonging to an edge is obtained as follows. Given the metric graph we can obtain a metric tree just cutting some  edges. Each edge can be cut in  different ways on a point parameterized by a real parameter. From the metric tree we can obtain a metric arborescence rooted at $x$ simply orienting all the edges towards $x$. To each metric arborescence we associate a weight. The weight is obtained as the product of several terms, one for each edge and one for each vertex. The weight on an edge is given by the exponential of $\int s_{\mathbf e}$ where the integral is an integral along the edge $\mathbf e$ according to its orientation on the arborescence and $s_\mathbf e:=\frac{b_\mathbf e}{\sigma^2_\mathbf e}$, where
$b_\mathbf e$ and $\sigma^2_\mathbf e$ are the drift and the diffusion coefficients on the metric edge $\mathbf e$. The weight associated to each node depends on the local orientation of the arborescence around the vertex (all apart one edge are oriented entering into the vertex) and the parameters describing the stochastic evolution on the node. In the continuous setting there is an extra factor, associated to the root $x$, that appears on the weight of each arborescence and is given by $\sigma_\mathbf e^{-2}(x)$.  The value of the invariant measure on each node is then adjusted depending again on the behaviour of the model on the vertices.

\smallskip

The strategy of the proof is the following. Before we compute the scaling limit of the combinatorial construction of the matrix tree Theorem for one dimensional diffusive random walks on a ring. We obtain that for any triple $(\alpha, Q, F)$ that correspond to the same $(b,\sigma)$ the limit is the same and coincides with the continuous construction above described. In this case a metric arborescence is obtained from the ring by just one single cut. The proof of this universal scaling limit is short and informal. Once obtained the guess on the form of the limiting construction we prove the invariance directly.

\smallskip
Models with equations and processes defined on metric graphs are used in several different applicative frameworks, like quantum mechanics \cite{QG}, traffic flow \cite{GP} and many other \cite{QG}. Our representation could be very useful in studying weak noise asymptotic \cite{FWb} of diffusions on metric graphs and in this limit it is very strongly connected with Hamilton Jacobi equations on metric graphs \cite{CM}.

\smallskip
We discussed the problem on a geometric framework that consists of several one dimensional spaces non trivially glued. We believe that a continuous version of the combinatorial construction exists in much more general frameworks, as for example for domains in $\mathbb R^d$ with $d\geq 2$.  
In this sense this paper is a first step toward the proof of a fascinating formula  like
\begin{equation}
\mu(x)= \frac{1}{\sigma^{2}(x)Z}\int\mathcal D(\tau_x)e^{\int_\Omega\left(b(y), \sigma^{-2}(y) v(y)\right)}\,.
\end{equation}
In the above equation $\Omega$ is a domain of $\mathbb R^d$ and $x\in \Omega$; $\int\mathcal D(\tau_x)$ is an integration over arborescenses rooted at $x$, the vector $v(y)$ denotes the direction at $y$ of the arborescence $\tau_x$ and $(\cdot ,\cdot)$ denotes the Euclidean scalar product. The challenge here is to give a meaning to all these objects.

There are also other  natural and interesting issues to be discussed in the framework of metric graphs. These are for example a determinantal representation and interpretation of the formulas, like in the discrete case; the relationship with the discrete time Markov chain obtained observing just the sequence of vertices visited; the connection with the theory of electrical networks on metric graphs \cite{QG}. We are not going to discuss here these issues.
\smallskip

The paper is organized as follows.

In Section 2 we shortly recall the classic Markov chain matrix tree Theorem for finite state irreducible continuous time Markov chains.

In Section 3 we discuss briefly the one dimensional diffusion processes. Then we recall the scaling limit of diffusive random walks giving the relation between the microscopic triple $(\alpha, Q, F)$ and the macroscopic pair $(b,\sigma)$. We discuss the scaling limit of the discrete arborescences and the corresponding weights getting formulas written in terms of continuous metric arborescences. Finally, we show by a direct computation that the representations obtained in the case of the circle and the interval give the correct result.

In Section 4 we describe shortly the metric graphs.

In Section 5 we discuss the definition of a diffusion process on a metric graph. This is done showing the form of the generator that depends on a collection of parameters related to the behaviour of the process in correspondence of the vertices.

In Section 6, we prove the validity of the representation formula for the invariant measure of a diffusion on a metric graph in terms of a normalized sum of weights associated to continuous metric arborescences obtained cutting the original metric graph on a finite number of points. We show moreover that the reversibility condition corresponds to the reversibility of a finite state effective Markov chain evolving on the vertices of the graph.

\section{Markov chain tree Theorem}
\label{MCTT}

Here we briefly recall a classic representation of the invariant measure of a finite irreducible Markov chain (see for example \cite{FWb} or \cite{PW} for a recent general overview or \cite{AT} for a simple proof).
The general framework is the following. We consider a Markov chain having transition graph $(V,E)$.
The finite set $V$ is the state space while the set of directed edges $E$ represents the collection of possible transitions of the chain. If $(x,y)\in E$ we have that the rate $r(x,y)$ of jump from $x$ to $y$ is strictly positive. We say that the directed edge $(x,y)$ exits from $x$ and enters into $y$.

The stationary condition that the invariant measure $\pi$ has to satisfy is given by
\begin{equation}\label{rice}
\pi(x)\sum_{y: (x,y)\in E} r(x,y)=\sum_{y:(y,x)\in E} \pi(y)r(y,x)\,, \qquad \forall x\in V\,.
\end{equation}
If the transition graph $(V,E)$ is strongly connected (any two points can be connected by directed paths) then the chain is irreducible and the invariant measure is unique and strictly positive. We restrict to this case.

\begin{definition}
Let $(V,E)$ be a directed graph. An \emph{arborescence} $\tau$ directed toward $x\in V$ is a spanning subgraph of $(V,E)$ such that:

\smallskip

\noindent 1) For each vertex $y\neq x$ there is exactly one directed edge exiting from $y$ and belonging to $\tau$;

\noindent 2) For any $y\in V$ there exists one directed path from $y$ to $x$ in $\tau$;

\noindent 3)  There are no edges exiting from $x$.

\smallskip

\noindent Let  $\mathcal T_x$ the set of \emph{arborescences} of $(V,E)$ directed toward $x\in V$.
\end{definition}

If the chain is irreducible, then $\mathcal T_x$ is not empty for any $x$.

Equivalently the arborescences in $\mathcal T_x$ can be characterized as follows. Take the transition graph $(V,E)$ and construct the corresponding undirected graph $(V,\mathcal E)$ obtained simply transforming each directed edge into an undirected one and removing all the multiple undirected edges obtained. An element $\tau\in \mathcal T_x$ is characterized by the fact that if we ignore orientation of the edges of $\tau$ we obtain a spanning tree of $(V,\mathcal E)$. Moreover, any directed edge in $\tau$ exiting from $y\neq x$ is directed according to the orientation obtained going along the unique un-oriented path from $y$ to $x$.

To any arborescence $\tau$ we associate a weight given by
\begin{equation}
\label{wfw}
R(\tau):=\prod_{e\in \tau} r(e)\,,
\end{equation}
where $r:E\to(0,\infty)$ are the transition rates.
In the above formula the product is over all the directed edges $e$ that are edges
of the arborescence $\tau$.

The Markov chain matrix tree Theorem claims that the invariant measure of the chain is given by
\begin{equation}\label{eq: FWformula}
\mu(x)=\frac{\sum_{\tau\in \mathcal T_x}R(\tau)}{\sum_{z\in V}\sum_{\tau\in \mathcal T_z}R(\tau)}\,.
\end{equation}
See \cite{PW} for a interpretation in terms of determinants of \eqref{eq: FWformula}.

\section{One dimensional diffusions}

Let $I$ be a finite interval  in $\mathbb{R}$ and let  us consider a one dimensional diffusion process determined by a stochastic differential equation
\begin{equation}
\label{diff}
dX_t=b(X_t)dt+\sigma(X_t)dW_t
\end{equation}
where $b$ is a  $C^1$-Lipschitz function, called \emph{drift},  and $\sigma$ is a $C^2$-Lipschitz function, called \emph{diffusion coefficient}.
Under these assumptions we have that the invariant measure $\mu=\mu(x)dx$ has a $C^2$-density that is a strong solution to
\begin{equation}\label{stazdiff}
\frac 12\partial^2_x\left(\sigma^2(x)\mu(x)\right)-\partial_x\left(b(x)\mu(x)\right)=0\, ,
\end{equation}
where $x$ belongs to the interior of $I$.
We assume strong regularity of the coefficients since we concentrate on the geometric construction of the invariant measures and try to minimize the unrelated technical details.
Equation \eqref{stazdiff} can be  very naturally written as follows. Consider a measure with  $C^2$-density $\nu=\nu(x)dx$
and define the corresponding \emph{probability current} as
\begin{equation}\label{curr}
J(\nu):=-\frac 12\partial_x\left(\sigma^2(x)\nu(x)\right)+b(x)\nu(x)\,.
\end{equation}
The stationary condition for the invariant measure $\mu$ can be written in terms of the current \eqref{curr} in one of the two equivalent conditions
\begin{equation}\label{divj}
\partial_xJ(\mu)=0\,, \qquad \Leftrightarrow \qquad  J(\mu)=\textrm{constant}\,.
\end{equation}
Let us call
\begin{equation}\label{defS}
s(x):=\frac{b(x)}{\sigma^2(x)}\,;\qquad S(x):=2\int_{x^*}^xs(y)dy\,,
\end{equation}
where $x^*$ is an arbitrary point.
The general solution to \eqref{stazdiff} on $\mathbb R$ is
\begin{equation}\label{ccc}
\mu(x)=\frac{1}{\sigma^2(x)}\left[k_1+k_2\int_{x^*}^xe^{-S(y)}dy\right]e^{S(x)},
\end{equation}
where $k_i$ are arbitrary constants that have to determined in order to obtain the invariant measure of \eqref{diff}. This procedure depends on the special geometrical frameworks and boundary conditions that we consider.

\smallskip

\paragraph*{\textbf{Example 1:} } Let us first consider the special case when the process \eqref{diff} is defined
on an interval $[a,b]$ with reflecting boundary conditions (see forthcoming Section \ref{forte} for a detailed discussion of boundary conditions). Since there is no flow
across the boundaries, the stationary condition \eqref{divj} coincides with $J(\mu)=0$, the process is always reversible and we easily get
\begin{equation}\label{facile}
\mu(x)=\frac{e^{S(x)}}{Z\sigma^2(x)}\,,
\end{equation}
where $Z$ is a normalization factor.
This means that we have to fix in \eqref{ccc} $k_2=0$ and the value of $k_1$ is then fixed by the normalization condition.

\smallskip

\paragraph*{\textbf{Example 2:} } Consider now the process \eqref{diff} on a ring of length one  $\mathcal S^1:=\mathbb R/\mathbb Z$. This is equivalent to fix the coefficients in \eqref{diff} periodic of period one. Given $z\in \mathbb R$ we denote by $\pi(z)\in \mathcal S^1$ its projection. We draw $\mathcal S^1$ as a ring on which
the anticlockwise direction corresponds to the direction of the  motion of  $\pi(z+t)$ for increasing $t$.

\smallskip

Given $x\neq y\in \mathcal{S}^1$ we call
$I^\pm[x,y]$ the closed intervals containing the points of $\mathcal{S}^1$ encountered
moving on the ring from $x$ to $y$ respectively  anticlockwise for the $+$ sign and clockwise
for the $-$ sign.

We will use integrals over oriented intervals of $\mathcal S^1$ so that our intervals will be always oriented. In particular the intervals $I^\pm[x,y]$ have always the orientations from $x$ to $y$. We have therefore that $I^+[x,y]$ and $I^-[y,x]$ have the same support but opposite orientations since $I^+[x,y]$ is anticlockwise oriented while $I^-[y,x]$ is clockwise oriented. When we write $\mathcal S^1$ we mean always the ring anticlockwise oriented.

Consider $I$ an oriented interval of $\mathcal S^1$ with extrema $x,y$ and orientation from $x$ to $y$. Let $z_1, z_2\in \mathbb R$ such that $x=\pi(z_1)$ and $z_2>z_1$ is the minimal element of $\mathbb R$ such that $y=\pi(z_2)$. If the orientation of the interval $I$ is the anticlockwise one then we define for a function $f:\mathcal S^1\to \mathbb R$ the integration over the oriented interval as
\begin{equation}\label{sembrainutile}
\int_I f(s)ds:=\int_{z_1}^{z_2}f(z)dz\,.
\end{equation}
If $\bar I$ is an interval of $\mathcal S^1$ having the same support of $I$ but opposite orientation then we define
\begin{equation}\label{sembraovvia}
\int_{\bar I} f(s)ds=-\int_I f(s) ds\,.
\end{equation}

\smallskip

The diffusion process \eqref{diff} on $\mathcal S^1$ may be non-reversible.
The condition of reversibility (see for example \cite{BGL}) is $s(x)=\nabla G(x)$ where $G$ is a function defined on $\mathcal{S}^1$ i.e. a periodic function of period $1$.
This condition is equivalent to have (recall definitions \eqref{defS})
\begin{equation}\label{bildetdif}
\int_{\mathcal{S}^1}s(x)dx=0\,.
\end{equation}
In particular the function $S$ defined in \eqref{defS} can be interpreted as a function on the ring $\mathcal{S}^1$ just in the reversible case. In the reversible case the stationarity condition becomes $J(\mu)=0$ and the invariant measure coincides
with \eqref{facile}. In the non-reversible case we have a more complicated solution. We will show that in this case the invariant measure can be very naturally represented by a continuous version of the combinatorial construction illustrated in Section \ref{MCTT}. We will then generalize this representation to arbitrary metric graphs.

\subsection{Scaling limit}

We discuss informally the diffusive scaling limit of a random walker. This is done to obtain the invariant measure of a diffusion process as the scaling limit of the invariant measure of the discrete walker. The computations for the scaling limit will be short and informal since once obtained
the limiting form of the measure we can prove directly that this is the correct one. The aim of the computation is to show that the basic structure of the combinatorial representation of Section \ref{MCTT} is preserved in the limit
getting a continuous version of the construction. This fact is shown in particular on a one dimensional ring.
\smallskip

Since we want a diffusive scaling limit we need to consider reversible random walks.
In particular we consider the most general reversible nearest neighbor random walk on the discrete circle with $N$ sites that we consider embedded into $\mathcal{S}^1$ with mesh $N^{-1}$. According to \cite{ACCG} this is determined by a weight function $\alpha_N :V\to \mathbb R^+$ and a weight function $Q_N:E\to \mathbb R^+$ such that $Q_N(x,y)=Q_N(y,x)$. A random walk is reversible if and only if the jump rates are fixed by
\begin{equation}\label{rates-rev}
r_N(x,y):=\alpha_N(x)Q_N(x,y)\,, \qquad y=x\pm \frac{1}{N}\,.
\end{equation}

We consider the case when the weight function $\alpha_N$ is obtained by the discretization of a $C^2$-function $\alpha:\mathcal{S}^1\to \mathbb R^+$ by fixing $\alpha_N(x):=\alpha(x)$ when $x\in V$.
The weight function $Q_N$ is likewise obtained by the discretization of a $C^2$-function $Q:\mathcal{S}^1\to \mathbb R^+$ by fixing $Q_N(x,y):=Q\left(\frac{x+y}{2}\right)$. We assume that $\alpha$ and $Q$ are strictly positive.

We can allow to perturb these rates switching on a weak external field. The scaling behavior stays again diffusive. We consider a $C^1$-vector field  $F$ on $\mathcal{S}^1$. This is a $C^1$-periodic function $F:\mathbb R\to \mathbb R$. The discretized version is a discrete vector field $F_N:E\to \mathbb R$ defined by
\begin{equation}\label{dische}
F_N(e):=\int_{e}F(x) dx\,, \qquad e\in E\,,
\end{equation}
where the integral in \eqref{dische} is the  integral on the oriented segment
going from the tail of the directed edge $e\in E$ to its head.
By definition we have $F_N(x,x+\frac{1}{N})=-F_N(x+\frac{1}{N},x)$.

The perturbed rates \cite{ACCG} are defined by
\begin{equation}\label{copiano}
r^F_N(x,y):=r_N(x,y)e^{F_N(x,y)}\,.
\end{equation}
We stress the difference between the two discretizations for $F_N$ and $Q_N$. We have indeed that $F_N(x,y)=O(N^{-1})$ while instead $Q_N(x,y)=O(1)$.

\smallskip

According to the general discussion in \cite{ACCG}, when the lattice is of mesh $1/N$ and the rates are rescaled by a factor of $N^2$ (diffusive rescaling) we obtain that the law of the random walk converges to the law of a diffusion process with a forward Kolmogorov equation for the evolution of the probability measures given by
\begin{equation}\label{ilgen}
\partial_t\mu_t(x)=-\partial_xJ(\mu_t)=\partial_x\Big[Q(x)
\partial_x\Big(\alpha(x)\mu_t(x)\Big)\Big]-2\partial_x\Big[\alpha(x)Q(x)F(x)\mu_t(x)\Big]\,.
\end{equation}
Comparing \eqref{ilgen} with \eqref{curr} we obtain the relation between the macroscopic parametrization of the diffusion process, that is given by $(b,\sigma)$ and the microscopic one that is determined by $(\alpha, Q, F)$. We obtain
\begin{equation}\label{arpane}
\left\{
\begin{array}{ll}
\sigma=\sqrt{2Q\alpha}\\
b=\alpha(2QF+\partial_x Q)\,.
\end{array}
\right.
\end{equation}
Let us discuss the inverse transformations of \eqref{arpane}. The first equation gives $\alpha=\frac{\sigma^2}{2Q}$. If we insert this in the second one we get
\begin{equation}\label{arpav}
\frac{\partial_x Q}{2Q}=s-F\,.
\end{equation}
Since the left hand side of \eqref{arpav} is a total derivative, i.e. its integral on the circle vanishes, we have that if a triple $(Q,\alpha, F)$ satisfies
\eqref{arpane} then we have
\begin{equation}\label{condE}
\int_{\mathcal S^1}F(x)dx= \int_{\mathcal S^1}s(x)dx\,.
\end{equation}
Once an external field $F$ satisfying \eqref{condE} has been fixed then the weights $Q, \alpha$ are uniquely determined as
\begin{equation}\label{eQa}
\left\{
\begin{array}{l}
Q(x)=ce^{[S+2 V](x)}\,,\\
\alpha(x)=\frac{\sigma^2(x)}{2c}e^{-[S+2V](x)}\,.
\end{array}
\right.
\end{equation}
Where $c$ is an arbitrary positive constant and
$$
[S+2V](x):=\int_{I^\pm[x^*,x]}2\left(s(z)-F(z)\right)dz\,,
$$
where $x^*\in \mathcal S^1$ is an arbitrary point and the sign $\pm$ can be chosen arbitrarily since the result does not change due to \eqref{condE}.

\begin{remark} The above computations are independent of the boundary conditions and they hold also in the case that the lattice is embedded on an interval with suitably boundary conditions. In this case the external field $F$ can be fixed arbitrarily since for any external field the equation
\eqref{arpav} can be solved in $Q$ on an interval. Given an arbitrary external field, If we call $V(x):=-\int_{x^*}^xF(y)dy$ and  $S(x):=2\int_{x^*}^xs(y)dy$
where $x^*$ is an arbitrary point of the interval, then the weights $Q, \alpha$ are uniquely determined up to the choice of an arbitrary positive constant $c$ as in \eqref{eQa}.
\end{remark}

\subsubsection{Scaling limit of the invariant measure on the ring}

Recall that the discrete
walker is evolving on a discrete ring of mesh $1/N$ embedded into $\mathcal S^1$ and
that we draw the lattice as a ring on which the anticlockwise orientation corresponds to move from site $x$ to site $x+\frac{1}{N}$.

We will use in the discrete setting the following notation similar to the continuous one. Given $x\neq y\in V$, with the symbol $I^+[x,y]$ we mean the subgraph of $(V,E)$ containing the vertices and the directed edges  that are visited by a walker moving from $x$ to $y$ anti-clockwise, $x$ and $y$ included. Likewise  with the symbol $I^-[x,y]$ we mean the subgraph of $(V,E)$, $x$ and $y$ included, containing the vertices and the directed edges that are visited by a walker moving from $x$ to $y$ clockwise. Note that according to our definition the vertices belonging to $I^+[x,y]$ and $I^-[y,x]$ are the same but they contain oppositely oriented edges.

We introduce the following notation
\begin{equation}\label{ipesi}
R^\pm_N(x,y):=\left\{
\begin{array}{ll}
\prod_{e\in I^{\pm}[x,y]}r_N^F(e) & y\neq x\,,\\
1 & y=x\,.
\end{array}
\right.
\end{equation}
For the rates \eqref{copiano} we have that \eqref{ipesi} becomes
\begin{equation}\label{ascsol}
R^{\pm}_N(x,y)=\alpha^{-1}(y)\left(\prod_{z\in I^\pm [x,y]}\alpha(z)\right)\left(\prod_{e\in
	 I^\pm[x,y]}Q(e)\right)e^{\int_{I^\pm[x,y]}F(z)dz}\,.
\end{equation}

Using the matrix tree Theorem discussed in Section \ref{MCTT} we can write the invariant measure for the walker on the ring of mesh $1/N$ as
\begin{equation}\label{mun}
\mu^N(x)=\frac{1}{Z_N}\sum_{y\in V}R_N^{+}(y+\frac{1}{N},x)R_N^{-}(y,x)\,,
\end{equation}
where $Z_N$ is a normalization factor. Using \eqref{ascsol}, for any $y$ we have
\begin{multline}
R^{+}_N\left(y+ \frac{1}{N},x\right)R^{-}_N(y,x) =  \left(\prod_{z\in V}\alpha(z)\right) \left(\prod_{e\in E}\sqrt{Q(e)}\right) \\
 \frac{e^{\int_{I^+[y+\frac{1}{N},x]}F(z)dz}e^{\int_{I^-[y,x]}F(z)dz}}{\alpha(x)Q(y,y+\frac{1}{N})}\,.
\end{multline}
Using the above formula we obtain therefore that \eqref{mun} converges to
\begin{equation}\label{raptorus}
\mu(x)=\frac{1}{Z'\alpha(x)} \int_{\mathcal S^1}dy \frac{e^{\int_{I^+(y,x)}F(z)dz}e^{\int_{I^-[y,x]}F(z)dz}}{Q(y)}\,,
\end{equation}
where $Z'$ is a suitable normalization constant.

\smallskip

Formula \eqref{raptorus} is written in terms of the parameters of the microscopic walker. Using relation \eqref{arpane} and its inverse we can
can show that  for any triple $(\alpha, Q, F)$ corresponding to a given $(b,\sigma)$ formula \eqref{raptorus} coincides with
\begin{equation}\label{raptorusb}
\mu(x)=\frac{1}{Z\sigma^2(x)} \int_{\mathcal S^1}dy\,  e^{\left[\int_{I^+[y,x]}s(z)dz+\int_{I^-[y,x]}s(z)dz\right]}\,,
\end{equation}
where $Z$ is a normalization constant.
Indeed according to \eqref{arpav} we have
\begin{multline}
e^{\int_{I^\pm(y,x)}F(z)dz}=e^{\int_{I^\pm(y,x)}s(z)dz-\int_{I^\pm(y,x)}\frac{\partial_z Q}{2Q}(z)dz}= e^{\int_{I^\pm(y,x)}s(z)dz}\Big(\frac{Q(y)}{Q(x)}\Big)^{\frac{1}{2}}\,\nonumber
\end{multline}
and \eqref{raptorusb} is obtained.

\smallskip

The geometric interpretation of \eqref{raptorusb} is very clear. Fix a point $x\in \mathcal S^1$ where we want to compute the density of the invariant measure. The density is then obtained summing some weights over all possible continuous arborescences of $\mathcal S^1$
directed toward $x$. A continuous directed arborescense is obtained cutting $\mathcal S^1$ on a point $y\in \mathcal S^1$  and orienting the segments  toward $x$ (see Figure \ref{cerchiotaglio} and Section \ref{MTG} for precise definitions). The weight of the oriented continuos arborescence is obtained multiplying by a factor of $\sigma^{-2}(x)$ the exponential of the sum of  integrals of the form $\int s $ over the oriented segments. The construction is therefore a direct continuous generalization of the discrete construction apart the appearance of the factor $\sigma^{-2}(x)$ related to the position of the root that it is not present in the discrete case.
\begin{figure}
	
	\centering
	
	\includegraphics{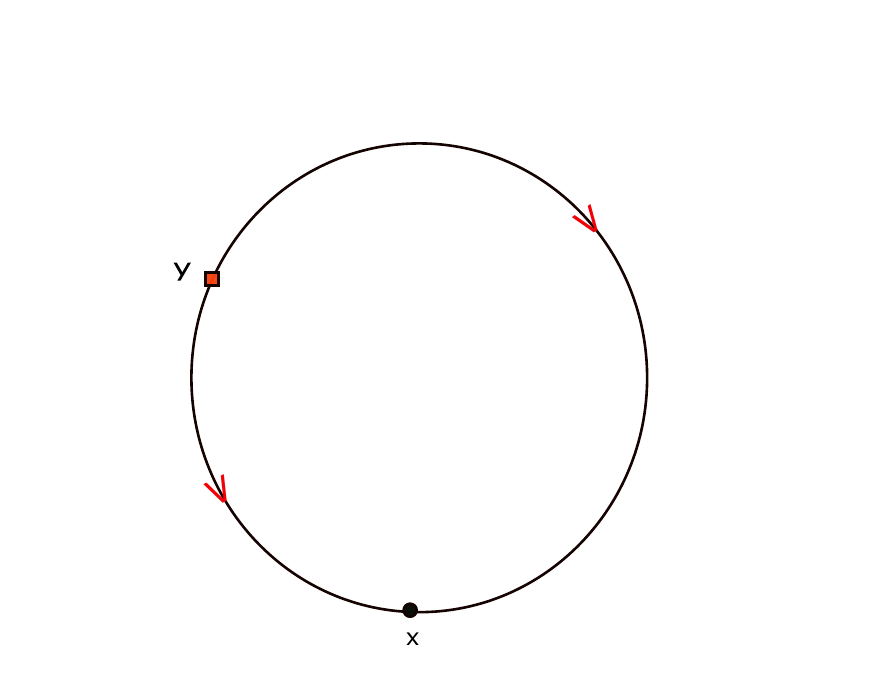}
	
	\caption{A continuous arborescence oriented toward the point $x$ obtained cutting $\mathcal S^1$ on the point $y$ and orienting $I^+[y,x]$ and $I^-[y,x]$ from $y$ to $x$.}\label{cerchiotaglio}
	
\end{figure}

\begin{remark}
Formula \eqref{raptorusb} can be written also like
\begin{equation}\label{wow}
\mu(x)=\frac{1}{Z\sigma^2(x)}\int_x^{x+1}e^{S(x)-S(y)}dy\,,
\end{equation}
that is a generalization of formula (2.3) in \cite{FG} that is particularly useful in the small noise limit.
\end{remark}

\subsection{Direct proof}

We computed shortly the scaling limit of the Markov chain matrix tree Theorem. We give now a direct proof that the obtained formula is the invariant measure of the limiting diffusion process. This is an elementary fact that will be however used in the following and similar computations will be relevant in the more general case. Once again we stress that the importance of formula \eqref{raptorusb} is indeed in the geometric interpretation.

\begin{lemma}\label{lemma31}
The unique invariant measure of the diffusion process \eqref{diff} on $\mathcal S^1$ is given by \eqref{raptorusb}.
\end{lemma}

\begin{proof}
We prove this fact by a direct computation. Uniqueness is classic, see for example \cite{BGL}. A key property is that when $y\neq x$ we have
\begin{equation}
\partial_x\left[\int_{I^+[y,x]}s(z)dz\right]=\partial_x\left[\int_{I^-[y,x]}s(z)dz\right]=s(x)\,.
\end{equation}
By this computation we deduce that if we call $\psi(y,x)$ the integrand in \eqref{raptorusb}, then this function is differentiable when $x\neq y$ and we have
\begin{equation}
\partial_x\psi(y,x)=2s(x)\psi(y,x)\,, \qquad x\neq y\,.
\end{equation}
Moreover $\psi$ has a discontinuity at $y=x$ given by
\begin{equation}\label{dd}
\Delta\psi(x,x)=\lim_{\epsilon \downarrow 0}\psi(x-\epsilon, x)-\psi(x+\epsilon, x)=
e^{-\int_{\mathcal S^1}s(z)dz}-e^{\int_{\mathcal S^1}s(z)dz}\,.
\end{equation}
Recall that as usual in the above formula we considered $\mathcal S^1$ anticlockwise oriented.
Note that the jump at the discontinuous points does not depend on $x$ and therefore we call just $\Delta\psi$ the right hand side of \eqref{dd}.
We can proceed as follows
\begin{align*}
&\partial_x\left(\int_{\mathcal S^1}\psi(y,x)dy\right)\\
&=\partial_x\left(\int_{I^+[x+\epsilon,x-\epsilon]}\psi(y,x)dy\right)+\partial_x\left(\int_{I^+[x-\epsilon,x+\epsilon]}\psi(y,x)dy\right)\\
&=2s(x)\left(\int_{I^+[x+\epsilon,x-\epsilon]}\psi(y,x)dy\right)
+\psi(x-\epsilon, x)-\psi(x+\epsilon, x)\\
&+\partial_x\left(\int_{I^+[x-\epsilon,x+\epsilon]}\psi(y,x)dy\right)\,.
\end{align*}
The last term in the above chain can be shown to be negligible in the limit $\epsilon\downarrow 0$ so that we deduce taking this limit on the right hand side of the above computation
\begin{equation}
\partial_x\left(\int_{\mathcal S^1}\psi(y,x)dy\right)=2s(x)\left(\int_{\mathcal S^1}\psi(y,x)dy\right)+\Delta \psi\,.
\end{equation}
Inserting these computations into \eqref{curr} we obtain
\begin{equation}
J(\mu)=-\frac{\Delta\psi}{2Z}\,
\end{equation}
that is constant and represents the typical value of the current across the ring. We have therefore that \eqref{divj} is satisfied and this implies that \eqref{raptorusb} is the unique invariant measure of \eqref{diff} on $\mathcal S^1$. This is because  it is naturally periodic, it is not negative, normalized to one and satisfies \eqref{stazdiff}. \end{proof}

\smallskip

\begin{remark}
 In the case of an interval $[a,b]$ there are no cycles and once fixed the point $x$ where to compute the density there are no cuts to be done. The oriented arborescence is obtained just orienting the intervals like $[a,x]$ and $[b,x]$, see Figure \ref{intervallone}. We get from the scaling limit
\begin{equation}\label{facile2}
\mu(x)=\frac{1}{Z\sigma^2(x)}  e^{\left[\int_{a}^xs(z)dz+\int_{b}^xs(z)dz\right]}\,.
\end{equation}
Formula \eqref{facile} coincides with \eqref{facile2} since
\begin{equation}\label{facile3}
\int_{a}^xs(z)dz+\int_{b}^xs(z)dz=2\int_{x^*}^xs(z)dz+c\,,
\end{equation}
with $c$ a suitable additive constant.

\end{remark}
\begin{figure}
	
	\centering
	
	\includegraphics{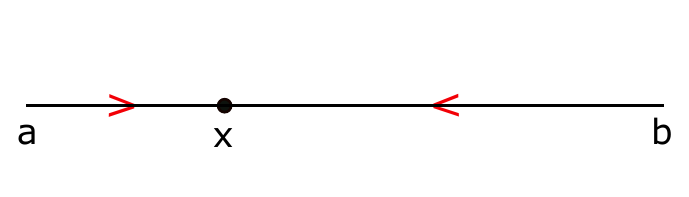}
	
	\caption{The continuous arborescence in the case of the interval. There are no cuts to be done and the orientation is drown with red arrows.}\label{intervallone}
	
\end{figure}

\section{Metric graphs and metric arborescences}\label{MTG}

We give a quick and informal description of a family of metric spaces called \emph{metric graphs}. We refer for more details for example to \cite{M} or \cite{BBI} Section 3.2.2.

We describe a finite metric graph $\mathcal G=(V,\mathbf E)$.
The set of vertices $V$ is a finite set. The set of metric edges $\mathbf E$ is a finite set too containing $|\mathbf E|$ metric edges. An element $\mathbf e\in \mathbf E$ is identified with an open interval $\mathbf e=(0,\ell_{\mathbf e})$, $\ell_{\mathbf{e}}\in\mathbb{R}$. A metric graph is a metric space obtained gluing the intervals associated to the edges to  the vertices in $V$. More precisely we glue the edge $\mathbf e\in \mathbf E$ to the vertex $v\in V$ identifying one of the two extrema of $\mathbf e$
with $v$. Note that, as explained in the following, the metric edges have an intrinsic orientation (the one corresponding to increasing coordinates) and the gluing can be done in two different ways identifying $v$ with the endpoint corresponding to the coordinate $0$ or to the coordinate $\ell_{\mathbf e}$. If we disregard orientation the two identifications are equivalent. We allow also for the third possibility when $v$ is identified with both endpoints of $\mathbf e$ obtaining a ring of length $\ell_e$ with the marked point $v$. When a vertex is identified  with one endpoint of a metric edge we say that the edge is incident to the vertex. Every endpoint of an edge is identified with exactly one vertex. We cannot have endpoints of edges not identified with any vertex.

The metric graph is obtained starting from the collection of vertices and metric edges and performing a finite number of identifications of vertices and endpoints according to the above rules.

The distance between two points on $\mathcal G$ is the length of the minimal path moving along the edges and going from one endpoint of one edge to an endpoint of another edge if both are identified with the same vertex.

\smallskip

If we disregard orientation of the metric edges and consider them just as segments of length $\ell_\mathbf e$ we can identify the metric structure of the metric graph giving just its combinatorial arrangement and the lengths.
More precisely given $v\in V$ we denote by $A^+(v),A^-(v)\subset \mathbf E$ the metric edges that are respectively exiting from the vertex $v$ and entering toward the vertex $v$. Moreover we define $A(v)=A^+(v)\cup A^-(v)$. Given a metric graph $(V,\mathbf E)$ we construct the corresponding un-oriented graph $(V, \mathcal E)$ defined as follows. The set of vertices  is again $V$ and given $v,w\in V$ we have that there is a number of un-oriented edges  $\{v,w\}\in \mathcal E$ equal to $\Big|A(v)\cap A(w)\Big|$. We have moreover a number of loops $\{v,v\}$ equal to $\Big|A^+(v)\cap A^-(v)\Big|$. The un-oriented graph is weighted and the edge $\{v,w\}\in \mathcal E$ corresponding to $\mathbf e\in \mathbf E$ has a weight given by $\ell_\mathbf e$.

The metric graph $(V, \mathbf E)$ is connected if $(V,\mathcal E)$ is connected and it is a metric tree if $(V, \mathcal E)$ is a tree.

\smallskip

Formally this can be equivalently formulated as follows. A path form $x\in (V , \mathbf E)$ to $y\in (V, \mathbf E)$ is a continuous map $\psi:[a,b]\to (V, \mathbf E)$, where $a<b$ are some given real parameters, and such that $\psi(a)=x$ and $\psi(b)=y$. A metric graph is connected if for any pair $x,y\in (V, \mathbf E)$ there exists a path from $x$ to $y$. A metric graph is called a metric tree if for any two points $x,y \in (V, \mathbf E)$ there is an unique injective path, up to reparametrizations, going from $x$ to $y$.

\smallskip

To draw pictures it is very useful to consider the graph embedded into $\mathbb R^d$ but the specific embedding is irrelevant. The vertices are therefore points
of $\mathbb R^d$. The edges are disjoint and not self-intersecting regular curves of length $\ell_\mathbf e$. Every edge connects two vertices of $V$ or just one in the case of loops. An edge $\mathbf e\in \mathbf E$ is parametrically described by the corresponding interval $(0,\ell_{\mathbf e})$, with $\ell_{\mathbf{e}}\in(0,\infty)$, and a $C^1$ map $\phi_{\mathbf e}:(0,\ell_{\mathbf e})\to \mathbb R^d$ such that $|\phi'_{\mathbf e}|(x)=1$  for any $x\in (0,\ell_\mathbf e)$. The vertices connected by the edge $\mathbf e$ are recovered by
$\lim_{x\downarrow 0}\phi_{\mathbf e}(x)$ and $\lim_{x\uparrow \ell_{\mathbf e}}\phi_{\mathbf e}(x)$.

\smallskip

We allow for the possibility of cut some of the edges of the metric graph. Given $\mathbf e\in \mathbf E$ and $x\in \mathbf e$ when we cut the metric graph at $x$ we remove the point $x$ from the metric graph $(V, \mathbf E)$. The new metric graph obtained after the cut is as follows. The structure remains unchanged for all the edges that do not contain the cutting point while the edge $\mathbf e$ containing $x$ is removed and substituted by two different edges
$\mathbf{e}_1=(0,x)$ and $\mathbf{e}_2=(0,\ell_\mathbf e-x)$. Here and hereafter with abuse of notation we call $x\in \mathbf e$ both the point on the edge and its coordinate
on the corresponding interval $(0,\ell_\mathbf e)$. The endpoint $0$ of the first edge is identified with the same vertex of the endpoint $0$ of $\mathbf e$. The endpoint $\ell_\mathbf e-x$ of the second edge is identified with the same vertex of the endpoint $\ell_\mathbf e$ of $\mathbf e$. Finally we add two new vertices $v_1, v_2$ and identify $v_1$ with the endpoint $x$ of $\mathbf{e}_1$ and $v_2$ with the endpoint $0$ of $\mathbf{e}_2$.

Note that before the cut the points of $\mathbf e$ with coordinates $x-\epsilon$ and $x+\epsilon$ for $\epsilon$ small enough are at distance $2\epsilon$ while this is no more the case after the cut.

\smallskip

Every edge of a metric graph is naturally oriented according to the increasing direction of the coordinates. The opposite orientation corresponds to the decreasing direction. There are just two possible orientations for each edge. An orientation of a metric graph is simply a choice between the two possible orientations for each edge. An orientation is
represented drawing an arrow on each edge. The canonical orientation of a metric graph is the one on which each edge is increasingly oriented.
Given a vertex $v\in V$, according to our notation, $A^+(v)$ and $A^-(v)$ are the sets of edges incident to $v$ that are respectively canonically oriented exiting from $v$ or entering into $v$.
A path $\psi:[a,b]\to (V, \mathbf E)$ is compatible with the orientations of the edges if when restricted to each edge the map $\psi$ is increasing or decreasing depending if the edge is increasingly or decreasingly oriented.

\smallskip

A \emph{metric arborescence} oriented toward $x\in \mathbf E$ is a metric graph that is a metric tree
and moreover all the edges are oriented toward the root $x$. This means that given any $y\in (V, \mathbf E)$ there exists an unique injective path from $y$ to $x$
and the path is compatible with the orientations. This essentially means that it is possible to reach the root $x$ starting from any other point $y$ and moving on the graph following the edges according to their orientation. Note that on a metric arborescence the edge containing the root is divided into two parts oppositely oriented.

Given a metric graph $\mathcal G$ we can obtain a metric arborescence cutting some of the edges and suitably orienting the edges, see Figure \ref{arbor} for an example and the following Section \ref{crim} for a more detailed description.

\begin{figure}
	
	\centering
	
	\includegraphics{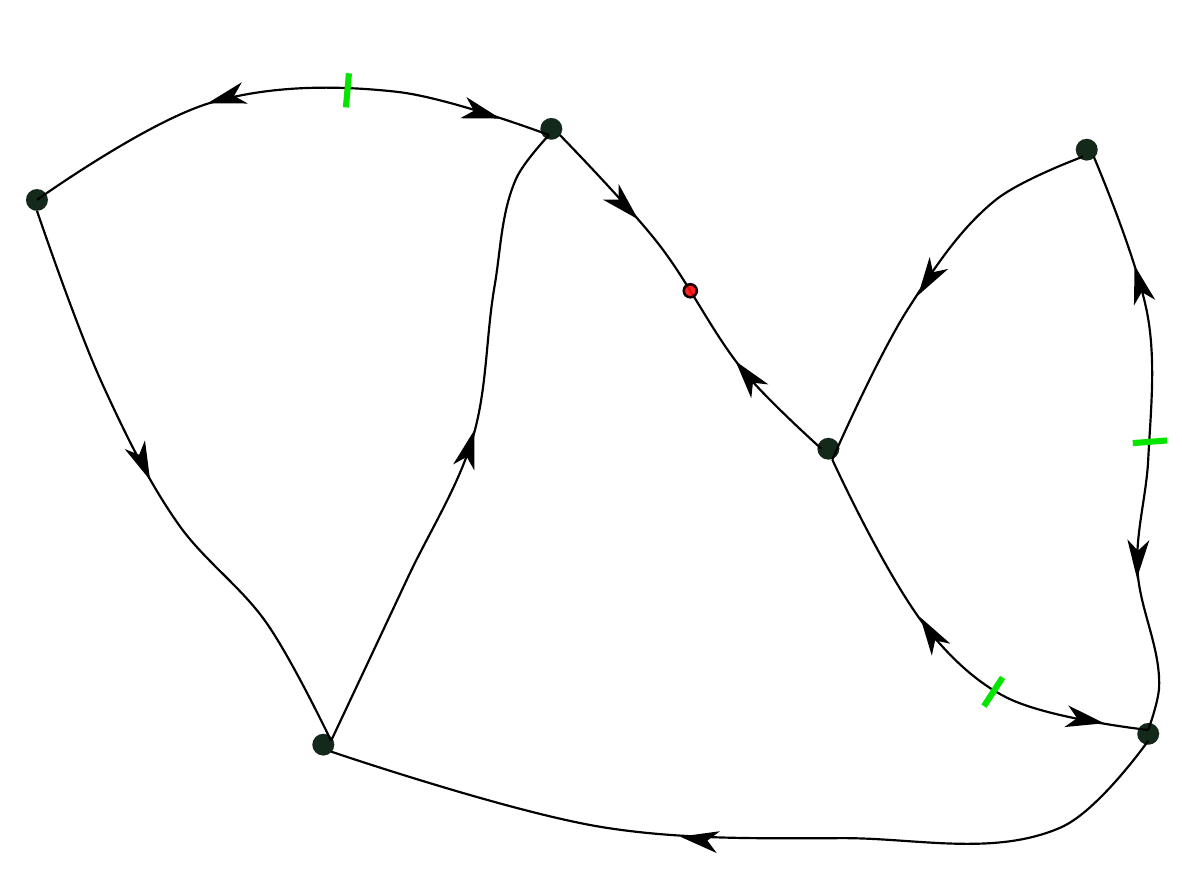}
	
	\caption{A metric graph embedded into $\mathbb R^d$ and an associated oriented arborescence obtained
	cutting the edges on the green slices. The root of the arborescence is drawn as a red dot.}\label{arbor}
	
\end{figure}

\section{Diffusions on metric graphs}\label{forte}

A walker is moving randomly on $\mathcal G=(V,\mathbf E)$ according to the following mechanism. For each edge $\mathbf e\in \mathbf E$ we fix coefficients  $(b_{\mathbf e}, \sigma_{\mathbf e})$ defined on the intervals $(0,\ell_{\mathbf e})$, $\ell_{\mathbf{e}}\in\mathbb{R}$, and again for simplicity we require $b_{\mathbf e}\in C^1$ and $\sigma_{\mathbf e}\in C^2$. The coefficients can be extended continuously on $[0,\ell_{\mathbf e}]$ and we require $\min_{x\in[0,\ell_{\mathbf e}] }\sigma_\mathbf e(x)>0$. When the walker is on the edge $\mathbf e$ then
the coordinate of the walker on $(0,\ell_{\mathbf e})$ evolves as a diffusion process $X_{\mathbf e}(t)$ satisfying the equation \eqref{diff}  with coefficients $(b_{\mathbf e}, \sigma_{\mathbf e})$.  When the walker reaches a vertex $v\in V$ then she can spend some time there and then the new edge on which she is continuing the evolution is chosen randomly  according to a given probability distribution on $A(v)$. The precise formulation of the dynamics can be formalized using the excursions of the diffusion processes, see for example  \cite{BPY,W}. In the following we will use an alternative approach to define the dynamics considering the corresponding generator.

We illustrate a formal definition of the diffusion processes on a metric graph describing the structure of the generator. We state the basic facts and refer to \cite{FS,FW1,FW2} for a general discussion. We give a short and essential summary of the features of the possible dynamics, all of them depend on the behavior of the walker on the nodes.

All the possible Feller processes with continuous sample path and such that inside each edge $\mathbf e\in \mathbf E$ the random dynamics coincides with a diffusion with parameters $(b_\mathbf e, \sigma_\mathbf e)$ are determined by a family of nonnegative parameters $\alpha$'s. In particular we have a family of parameters associated to the vertices $\left(\alpha_v\right)_{v\in V}$ and a family of parameters associated to the pairs $(v,\mathbf e)$ such that $v\in V$ and $\mathbf e\in A(v)$. We denote them as $\left(\alpha_{v,\mathbf e}\right)_{v\in V,\mathbf e\in A(v)}$.
The parameters must satisfy the constraints
\begin{equation}
\alpha_v+\sum_{\mathbf e\in A(v)}\alpha_{v,\mathbf e}>0\,, \qquad \forall \ v\in V\,.
\end{equation}
Once a collection of parameters is fixed then the process is defined by its generator $\mathcal A$ that is a linear operator on a suitable subset of $C(\mathcal G)$, the set of continuous function on the metric graph $\mathcal G$. To describe the form of the generator we need some notation. To avoid problems of irreducibility we consider always metric graphs that are connected and we will assume that each coefficient $\alpha_{v,\mathbf e}$ is strictly positive.

A function $f\in C(\mathcal G)$ is determined by a family of continuous functions $f_\mathbf e:(0,\ell_\mathbf e)\to \mathbb R$. The value of the function $f$ on a point $x\in \mathbf e$ is given by $f(x):=f_\mathbf e\left(x\right)$. Here again with abuse of notation we call $x\in \mathbf e$ both the point and its coordinate in $(0,\ell_\mathbf e)$. In order that $f\in C\left(\mathcal G\right)$ we need to impose the condition that given any $v\in V$ there is a real number $f(v)\in \mathbb R$ (the value of the function at the vertex $v$) such that
\begin{equation}
f(v)=\lim_{x\uparrow \ell_\mathbf e} f_\mathbf e(x)=\lim_{x\downarrow 0} f_{\mathbf e'}(x)\,, \qquad \forall\ \mathbf e\in A^-(v)\,,\, \mathbf e'\in A^+(v)\,.
\end{equation}

Given $v\in V$, $\mathbf e\in A(v)$ and $f\in C(\mathcal G)$ we define $D_\mathbf e f(v)$ the exiting derivative of $f$ at $v$ along $\mathbf e$ as
\begin{equation}
D_\mathbf e f(v):=\left\{
\begin{array}{ll}
\lim_{x\downarrow 0}\frac{f_\mathbf e(x)-f(v)}{x}\,, & \mathbf e\in A^+(v)\,,\\
\lim_{x\uparrow \ell_\mathbf e}\frac{f_\mathbf e(x)-f(v)}{\ell_\mathbf e-x}\,, & \mathbf e\in A^-(v)\,,
\end{array}
\right.
\end{equation}
when the limits exist.

\smallskip

We define now the generator $\mathcal A$ of our dynamics.
\begin{definition}[\emph{Generator of the dynamics}]\label{def}
Consider $f\in C(\mathcal G)$ with $f_\mathbf e\in C^2(0,\ell_\mathbf e)$. For an $x\in \mathbf e \in \mathbf E$ we define
\begin{equation}
[\mathcal Af](x):=L_\mathbf e f_\mathbf e\left(x\right)\,,
\end{equation}
where $L_\mathbf e$ is the generator of the diffusion process with parameters $(b_\mathbf e, \sigma_\mathbf e)$, i.e.
\begin{equation}\label{gensue}
L_{\mathbf e} f_{\mathbf e}(x):= \frac12 \sigma^2_{\mathbf e}(x)\partial^2_xf_\mathbf e(x)+b_\mathbf e(x)\partial_xf_\mathbf e(x)\,, \qquad x\in (0,\ell_\mathbf e)\,.
\end{equation}
The operator $\mathcal A$ on the nodes is defined by the relation
\begin{equation}\label{nodecond}
\alpha_v [\mathcal Af](v)=\sum_{\mathbf e\in A(v)}\alpha_{v,\mathbf e} D_\mathbf e f(v)\,, \qquad v\in V\,.
\end{equation}
The domain of definition $\mathcal D(\mathcal A)\subseteq C(\mathcal G)$ of the operator $\mathcal A$
is the set of functions for which all the derivatives exist and such that $\mathcal A f\in C\left(\mathcal G\right)$.
\end{definition}

The following result classifies all the possible diffusions on metric graphs.
\begin{theorem} \cite{FS}, \cite{FW1}, \cite{FW2}\label{FW}
The operator $\mathcal A$ as defined in Definition \ref{def} is the infinitesimal generator of a strongly continuous semigroup of operators on $C(\mathcal G)$ corresponding to a conservative Markov process on $\mathcal G$ with continuous paths. The following statements hold:

\begin{itemize}
	\item Before the process leaves an edge $\mathbf e$ it evolves like a diffusion process with generator $L_\mathbf e$ in \eqref{gensue}
	
	\item If $\alpha_v=0$ then the process spends almost surely zero time at the vertex $v\in V$
	
	\item If $\alpha_v=0$ for any $v\in V$ then the distribution of the process at any given time has a density with respect to a measure that is zero on all the vertices $v\in V$
\end{itemize}
Conversely if $X(t)$ is a Feller conservative Markov process on $\mathcal G$ that coincides,  before leaving $\mathbf e$, with a diffusion having generator \eqref{gensue} then its infinitesimal generator coincides with $\mathcal A$ in Definition \ref{def} for a suitable choice of the parameters. Moreover if the process $X$ spend almost surely zero time at $v\in V$ then necessarily $\alpha_v=0$.
\end{theorem}

We discuss now the invariant measures of the diffusion processes on the metric graphs. The conditions that we get can be easily described in terms of a divergence free condition for the probability currents \eqref{curr} on the edges. Fix an arbitrary reference orientation for each edge that for us will be always the canonical one. For $v\in V$, $\mathbf e\in A(v)$  and a collection of functions $\left(g_\mathbf e(x)\,,\, x\in (0,\ell_\mathbf e)\right)_{\mathbf e \in \mathbf E}$ we denote by
\begin{equation}\label{vieris}
g_{\mathbf e}(v):=\lim_{y\in \mathbf e, y\to v}g_\mathbf e(y)\,.
\end{equation}

The invariant measures $\mu$ for this class of processes contain atomic components on the vertices and are absolutely continuous with respect to the Lebesgue measure on the edges. A measure of this type is denoted by
\begin{equation}\label{cms}
\mu=\Big\{\left(\mu_v\right)_{v\in V}, \left(\mu_\mathbf e(x)dx\right)_{\mathbf e\in \mathbf E}\Big\}\,.
\end{equation}
The number $\mu_v$ is the weight of the atomic component on the vertex $v\in V$ while $\mu_\mathbf e(x)dx$ is the density of the absolutely continuous component on the edge $\mathbf e\in \mathbf E$. The normalization condition is
\begin{equation}\label{norm}
\sum_{v\in V}\mu_v+\sum_{\mathbf e\in \mathbf E}\int_0^{\ell_\mathbf e}\mu_\mathbf e(x)dx=1\,.
\end{equation}
Given a measure $\mu$ with $C^1$ densities, we define a probability current $J[\mu]=\left(J_\mathbf e[\mu]\right)_{\mathbf e\in \mathbf E}$ on the edges of $\mathcal G$ by
\begin{equation}\label{defcurre}
J_\mathbf e[\mu]:=-\frac 12\partial_x\left(\sigma^2_\mathbf e(x)\mu_\mathbf e(x)\right)+b(x)\mu_\mathbf e(x)\,, \qquad x\in (0,\ell_\mathbf e)\,.
\end{equation}

\begin{lemma}\label{dot}
Consider a diffusion process on a metric graph $\mathcal{G}$ having generator $\mathcal A$ as in Definition \ref{def}. Then the process has a unique invariant measure $\mu$ of the form \eqref{cms} characterized by the following properties:
\begin{enumerate}\label{sipuo}
	\item The current $J_\mathbf e[\mu]$ is constant on each edge $\mathbf e\in \mathbf E$.
	
	\item On each node $v\in V$ we have the divergence free condition
	$$
	\div J[\mu](v):=\sum_{\mathbf e\in A^+(x) } J_\mathbf e[\mu]-\sum_{\mathbf e\in A^-(x) } J_\mathbf e[\mu]=0\,.
	$$
	
	\item There exist some constants $\lambda_v>0$ such that
	\begin{equation}\label{arrivanoirussi}
	\left\{
	\begin{array}{ll}
	\frac 12\sigma^2_\mathbf e(v)\mu_{\mathbf e}(v)=\lambda_v\alpha_{v,\mathbf e} & v\in V, \mathbf e\in A(v)\,,\\
	\mu_v=\lambda_v\alpha_v & v\in V\,.
	\end{array}
	\right.
	\end{equation}
\end{enumerate}
Furthermore the process is reversible if and only if conditions $(1)+(2)$ are substituted by the single condition
\begin{enumerate}
\item[(1')] $J_\mathbf e[\mu]=0$ on each edge $\mathbf e\in \mathbf E$\,.
\end{enumerate}
\end{lemma}
\begin{proof}
Uniqueness follows by \cite{Boga}, \cite{Risken}. The fact that the unique invariant measure is of the form \eqref{cms} follows by Theorem \ref{FW}. We show now that a probability measure $\mu$ satisfies the three conditions above if and only if
\begin{equation}\label{egia}
\int_{\mathcal{G}} \left[\mathcal Af\right] d\mu=0\,, \qquad \forall f\in \mathcal D(\mathcal A)\,.
\end{equation}

Consider indeed a generic $f\in \mathcal D(\mathcal A)$ and perform a double integration by parts on each edge. We obtain the following
\begin{align}
&\int_{\mathcal{G}} \left[\mathcal Af\right] d\mu= \sum_{v\in V}\left[ \mu_v[\mathcal A f](v)-\sum_{\mathbf e\in A(v)}
\frac12\sigma^2_\mathbf e(v)\mu_\mathbf e(v)D_\mathbf ef(v)\right]\\
&+\sum_{v\in V}f(v)\div J[\mu](v)-\sum_{\mathbf e\in \mathbf E}\int_0^{\ell_{\mathbf e}}f_\mathbf e(x)\partial_x J_\mathbf e(\mu)dx\,.\nonumber
\end{align}
Condition (1) implies that each term  of the third sum on the right hand side of the above equation is zero. Condition (2) implies the same for the second sum and condition (3)  for the first one. We have therefore that if the conditions are satisfied then \eqref{egia} holds. Conversely since the right hand side of the above equation has to be zero for each function $f\in \mathcal D(\mathcal A)$ then
each single term has to be identically zero and this implies the validity of the three conditions.

\smallskip

The reversibility is proved observing that, for all $f,h\in \mathcal D(\mathcal A)$, we get
\begin{align*}
\int_{\mathcal{G}} \left[\mathcal Af\right]h d\mu- \int_{\mathcal{G}} f\left[\mathcal Ah\right] d\mu =&\sum_{v\in V}\left[ \mu_v[\mathcal A f](v)-\sum_{\mathbf e\in A(v)}
\frac12\sigma^2_\mathbf e(v)\mu_\mathbf e(v)D_\mathbf ef(v)\right]g(v)\\
+ &\sum_{v\in V}\left[ \mu_v[\mathcal A g](v)-\sum_{\mathbf e\in A(v)}
\frac12\sigma^2_\mathbf e(v)\mu_\mathbf e(v)D_\mathbf eg(v)\right]f(v)\\
+ &\sum_{\mathbf e\in \mathbf E}\int_0^{\ell_{\mathbf e}}\big( \partial_x f_{\mathbf{e} }(x)h_\mathbf e(x) - f_\mathbf e(x)\partial_x h_{\mathbf{e}(x) }\big) J_\mathbf e(\mu)dx.
\end{align*}
The first two terms in the right hand side of the above formula are zero by (3).  The third term is zero by condition $(1')$. The converse statement is proved observing that the above equation has to be zero for any pair of functions $f,h\in \mathcal D(\mathcal A)$.
\end{proof}
\begin{remark}
Condition $(3)$ on Lemma \ref{dot} can be written as
\begin{equation}\label{c1}
\frac{\sigma^2_\mathbf e(v)\mu_\mathbf e(v)}{2\mu_v}=\frac{\alpha_{v,\mathbf e}}{\alpha_v}\,, \qquad v\in V, \mathbf e\in A(v)\,,
\end{equation}
for the vertices for which $\alpha_v>0$ and as
\begin{equation}\label{c2}
\frac{\sigma^2_\mathbf e(v)\mu_\mathbf e(v)}{\sigma^2_{\mathbf e'}(v)\mu_{\mathbf e'}(v)}=\frac{\alpha_{v,\mathbf e}}{\alpha_{v,\mathbf e'}}\,, \qquad \mathbf e, \mathbf e'\in A(v)\,,
\end{equation}
for the vertices for which $\alpha_v=0$. Note that for a vertex $v$ with $\alpha_v>0$ we have that condition \eqref{c1} implies condition \eqref{c2}.
\end{remark}
In the next section we construct explicitly a probability measure satisfying the conditions of Lemma \ref{dot} obtaining therefore the unique invariant measure of the process.

\section{A combinatorial representation of the invariant measure}
\label{crim}

We consider metric arborescences rooted at $x\in \mathcal G$ that we recall they are  oriented metric trees obtained cutting the original metric graph on a finite number of points and orienting all the edges toward the root (see Figure \ref{arbor} for an example). Let us call $\mathcal T_x$ the collection of all the metric arborescences rooted at $x\in \mathbf e\in \mathbf E$.

\smallskip

Given $\tau\in \mathcal T_x$ we define the corresponding weight as
\begin{equation}\label{pesogcal}
R(\tau):=\frac{e^{\int_{\tau}s}}{\sigma_{\mathbf{e}(x)}^2(x)}\prod_{v\in V}\mathcal W_v(\tau)\,.
\end{equation}
The symbol $\int_\tau s$ denotes the sum of all the integrals along the edges of the metric graph according to the orientations of $\tau$. In particular if the edge $\mathbf e$ does not contain a cut and it is oriented in $\tau$ according to its natural orientation then the contribution of the edge is given by $\int_0^{\ell_\mathbf e}s_\mathbf e(x)dx$. If instead the edge is oriented in $\tau$ oppositely with respect to the natural orientation then the contribution is
$-\int_0^{\ell_\mathbf e}s_\mathbf e(x)dx$. If the edge $\mathbf e$ contains a cut in the point $z\in (0,\ell_\mathbf e)$ then the contribution coming from this edge is given by $-\int_0^{z}s_\mathbf e(x)dx+\int_z^{\ell_\mathbf e}s_\mathbf e(x)dx$ since the orientation of the two branches of the edge are necessarily exiting from the cut-point. Recall that as before we use the notation $s_\mathbf e:=\frac{b_\mathbf e}{\sigma^2_\mathbf e}$.

The weight $\mathcal W_v(\tau)$ for the vertex $v\in V$ and the arborescence $\tau$ is defined by
\begin{equation}\label{cong}
\mathcal W_v(\tau):=\prod_{\mathbf e \in A(v) }W_v^{\theta_v(\mathbf e,\tau)}(\mathbf e)\,.
\end{equation}
In \eqref{cong} we have $\theta_v(\mathbf e,\tau)=\pm$ depending if $\mathbf e$ is oriented in the arborescence $\tau$ in a neighborhood of $v$ exiting or entering respectively into $v$. For any pair $v\in V$ and $\mathbf e \in A(v)$ we have therefore two free parameters $W^\pm_v(\mathbf e)>0$.

We fix $W^-_v(\mathbf e):=1$ for any edge $\mathbf e$, since this corresponds to multiplying the weights $R$ by a constant.

\smallskip

The number of cuts that have to be done in order to transform a metric graph $(V,\mathbf E)$ into a metric tree is fixed and it is given by $|\mathbf E|-|V|+1$. We call this number the \emph{dimension of the cut space}.
In order to obtain a metric tree the cuts have to be done on different edges, more precisely the exact procedure is the following. Associate to the metric graph $(V, \mathbf E)$ the un-oriented graph $(V, \mathcal E)$ as discussed in Section \ref{MTG}.  According to our definitions, loops and multiple edges are allowed. By construction we have $|\mathbf E|=|\mathcal E|$.
For any spanning tree $(V,T)$ of $(V,\mathcal E)$ we have that performing one cut
for any edge in $\mathbf E$ corresponding to an unoriented edge in $\mathcal E\setminus T$
we obtain a metric tree.

\smallskip

Let $\mathcal C$ the collection of cutting points that transform the metric graph $(V,\mathbf E)$ into a metric tree. When the dimension of the cut space is $k=|\mathcal E\setminus T|$ then a generic element of $\mathcal C$ is given by $\underline y=(y_1,\dots y_k)$ where each $y_i$ belongs to a different element of $\mathbf E$ associated to an edge in $\mathcal E\setminus T$. We call $\mathbb T$ the collection of the spanning trees of $(V,\mathcal E)$ and $\mathcal C(T)$ the cutting points compatible with the spanning tree $T\in \mathbb T$ as discussed above. We have therefore $\mathcal C=\cup_{T\in \mathbb T}\mathcal C(T)$. Given $\underline y=(y_1,\dots ,y_k)\in \mathcal C(T)$ we call $\tau_x[\underline y]$ the metric arborescence obtained cutting the metric graph on the points $(y_1, \dots ,y_k)$ and orienting the edges toward $x$.

Let us define a positive measure $m=\left(m_\mathbf e(x)dx\right)_{\mathbf e\in \mathbf E}$  on $\mathcal G$ that gives zero weight to vertices and that is absolutely continuous on the edges, defined by
\begin{equation}\label{laformula}
m_\mathbf e(x):=\sum_{T\in \mathbb T}\int_{\mathcal C(T)}dy_1\dots dy_k R\left(\tau_x\left[\underline y\right]\right)\,, \qquad x\in \mathbf e\,.
\end{equation}

We have the following
\begin{theorem}\label{ilteo}
The positive measure $m$ defined by \eqref{laformula} satisfies conditions (1), (2) and the equations in the upper line of condition (3) in Lemma \ref{sipuo} under the condition
\begin{equation}\label{cambiata}
W^+_v(\mathbf e)=K_v\alpha_{v,\mathbf e}\,,\qquad v\in V, \mathbf e\in A(v)\,,
\end{equation}
where $K_v>0$ is a family of arbitrary constants.
\end{theorem}
\begin{proof}
The basic fact to prove this Theorem is the computation of
\begin{equation}\label{calcolobase}
\partial_x\left(\int_{\mathcal C(T)}dy_1\dots dy_k e^{\int_{\tau_x[\underline y]}s}\right)\,.
\end{equation}
We do it distinguishing several cases.
The first case is when $x\in \mathbf e$ and $\mathbf e\cap \mathcal C(T)=\emptyset$. In this case we can differentiate directly and using the computations in Lemma \ref{lemma31} we get that
\eqref{calcolobase} is equal to
\begin{equation}\label{dacit}
\frac{2 b_{\mathbf e}(x)}{\sigma^2_{\mathbf e}(x)}\int_{\mathcal C(T)}dy_1\dots dy_k e^{\int_{\tau_x[\underline y]}s}\,.
\end{equation}
In the case instead that $x\in \mathbf e$ and $\mathbf e\cap \mathcal C(T)\neq \emptyset$ we have to do a different computation. Let us assume (for simplicity of notation and without loss of generality) that $y_k$ is the cutting point that belongs to $\mathbf e$. We can write the term to be differentiate in \eqref{calcolobase} as
\begin{equation}\label{lungai}
\int_{\mathcal C'(T)}dy_1\dots dy_{k-1} e^{\int_{\tau'_x[\underline y]}s}\left(\int_0^x dy_ke^{\int_{y_k}^0s}
e^{\int_{y_k}^xs}e^{\int_{\ell_{\mathbf e}}^xs}+
\int_x^{\ell_{\mathbf e}} dy_ke^{\int_{0}^xs}
e^{\int_{y_k}^xs}e^{\int^{\ell_{\mathbf e}}_{y_k}s}
\right).
\end{equation}
In the above formula we called $\mathcal C'(T)$ the cutting points that do not belong to $\mathbf e$ and we called $\tau'_x[\underline y]$ the arborescence $\tau_x[\underline y]$ without the oriented edges belonging to $\mathbf e$.
Differentiating \eqref{lungai} with respect to $x$ we have a term like \eqref{dacit} (that comes from the differentiation of the dependence on $x$ in the integrand) with  in addition a contribution (that comes differentiating the dependence on $x$ in the extrema of integration). This additional contribution is given by
\begin{equation}\label{daspiegare}
\int_{\mathcal C'(T)}dy_1\dots dy_{k-1} e^{\int_{\tau'_{x-}[\underline y]}s}e^{\int^{0}_xs}e^{\int_{\ell_{\mathbf e}}^xs}- \int_{\mathcal C'(T)}dy_1\dots dy_{k-1} e^{\int_{\tau'_{x+}[\underline y]}s}e^{\int^{\ell_{\mathbf e}}_xs}e^{\int_{0}^xs}\,.
\end{equation}
A few explanations are in order for formula \eqref{daspiegare}. Here there are 2 contributions, one is positive and the other one is negative. These contributions can be naturally interpreted as weights coming from arborescences that have a cut in the edge $\mathbf e $ that coincides exactly with the point $x$. The positive term corresponds to the weight of the arborescence when both parts of $\mathbf e$ are oriented in the opposite way with respect to the canonical one. The term with the minus sign corresponds instead to the situation when both parts of the edge $\mathbf e$ are oriented in agreement with the canonical orientation.
We call $\tau_{x\pm}[\underline y]$ the corresponding two different arborescences (see Figure \ref{taupiuomeno} for an illustrative example). In particular consider a cutting point $\underline y=(y_1,\dots ,x, \dots,y_k)$ having a cut in correspondence of $x\in \mathbf e$. We write $\tau'_{x-}[\underline y]$ to denote the part of the arborescence $\tau_{x-}[\underline y]$ outside of the edge $\mathbf e$ when  both parts of $\mathbf e$ are oriented in the opposite way with respect to the canonical one. We write instead $\tau'_{x+}[\underline y]$ to denote the part of the arborescence $\tau_{x+}[\underline y]$ outside of the edge $\mathbf e$ when both parts of $\mathbf e$ are oriented in agreement with respect to the canonical one.
\begin{figure}
	
	\centering
	
	\includegraphics{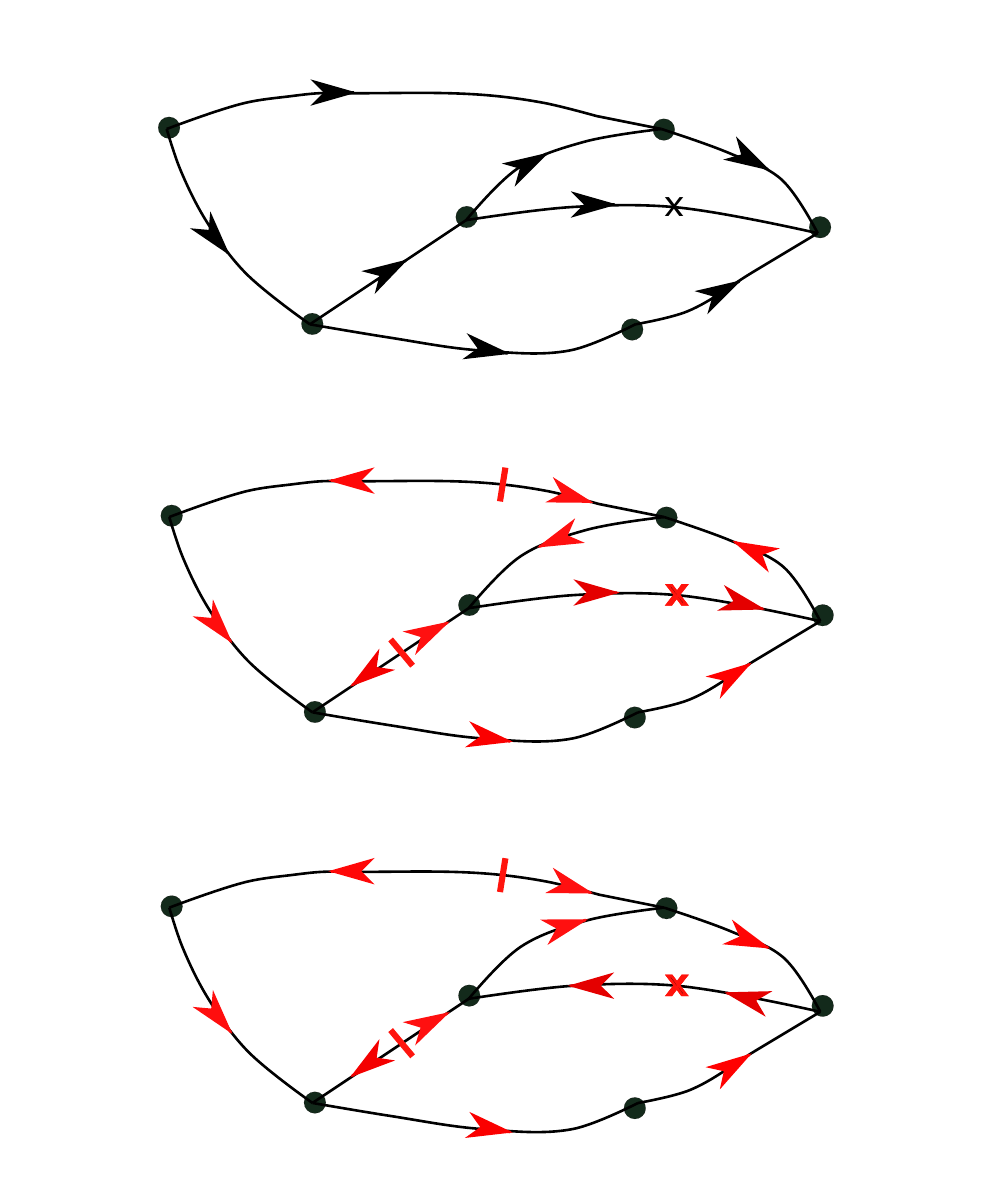}
	
	\caption{A metric graph with a marked point x. The black arrows denote the canonical orientation of the edges (top picture). The oriented arborescence $\tau_{x+}(\underline y)$ where the three cuts $\underline y=(y_1,y_2,x)$ are represented respectively by two cutting red segments and a red x in correspondence of the marked point x. The orientations of the branches of the arborescence are red colored (middle picture). The oriented arborescence $\tau_{x-}(\underline y)$ where $\underline y$ is as before. The orientations of the branches of the arborescence are red colored (bottom picture). }\label{taupiuomeno}
	
\end{figure}
\smallskip

An alternative natural interpretation  of the terms in \eqref{daspiegare} is the following. We call $\mathfrak L$ the set of uni-cyclic oriented metric spanning subgraphs of $(V,\mathbf E)$ defined as follows. We call $\mathbb L$ the connected spanning subgraphs of $(V,\mathcal E)$ that contains one single cycle and we call $\mathbb L_{\mathbf e}\subseteq \mathbb L$
the ones such that the element of $\mathcal E$ corresponding to $\mathbf e$ is one of the edges of the unique cycle. For any $L\in \mathbb L$ we call $\tilde{\mathcal{C}}(L)$ the sets of points $\tilde{y}=(y_1, \dots ,y_{k-1})\in \mathbf{E}^{k-1}$ such that the $y_i$ belong to different metric edges and each of them corresponds to an element of $\mathcal E$ that has been erased from $(V, \mathcal E)$ to get $L$.
We call a uni-cyclic spanning metric subgraph $\mathcal L\in\mathfrak L$ the metric graph obtained from $(V,\mathbb E)$ by the above cuts $(y_1,\dots ,y_{k-1})$, such that all the edges are oriented toward the unique cycle and the edges of the cycle  are oriented in such a way that it is possible to go around respecting the orientation. For each $\mathcal L\in \mathfrak L$ there exists a $\bar{ \mathcal L}\in \mathfrak L$ such that $\mathcal L$ and $\bar {\mathcal L}$ are obtained by the same cuts, the edges outside of the cycle are oriented in the same way and all the edges inside the cycle are oriented in the opposite way. Given $\tilde{y}\in \tilde{\mathcal{C}}(L)$ we call $\mathcal L(\tilde y,\theta)$, with $\theta=\pm$ the two uni-cyclic spanning metric subgraphs with the two possible orientations of the cycle. Consider a metric arborescence $\tau_{x\pm}(\tilde{y},x)$ obtained from $\mathcal G$ by a cutting set of the form $(\tilde{y},x)$. An uniciclyc metric graph $\mathcal L(\tilde{y},\theta)$ is obtained from $\tau_{x\pm}(\tilde{y},x)$ removing the cut at $x$ and in this case $x$ belongs to the unique cycle (see Figure \ref{unicicli}).
\begin{figure}
	
	\centering
	
	\includegraphics{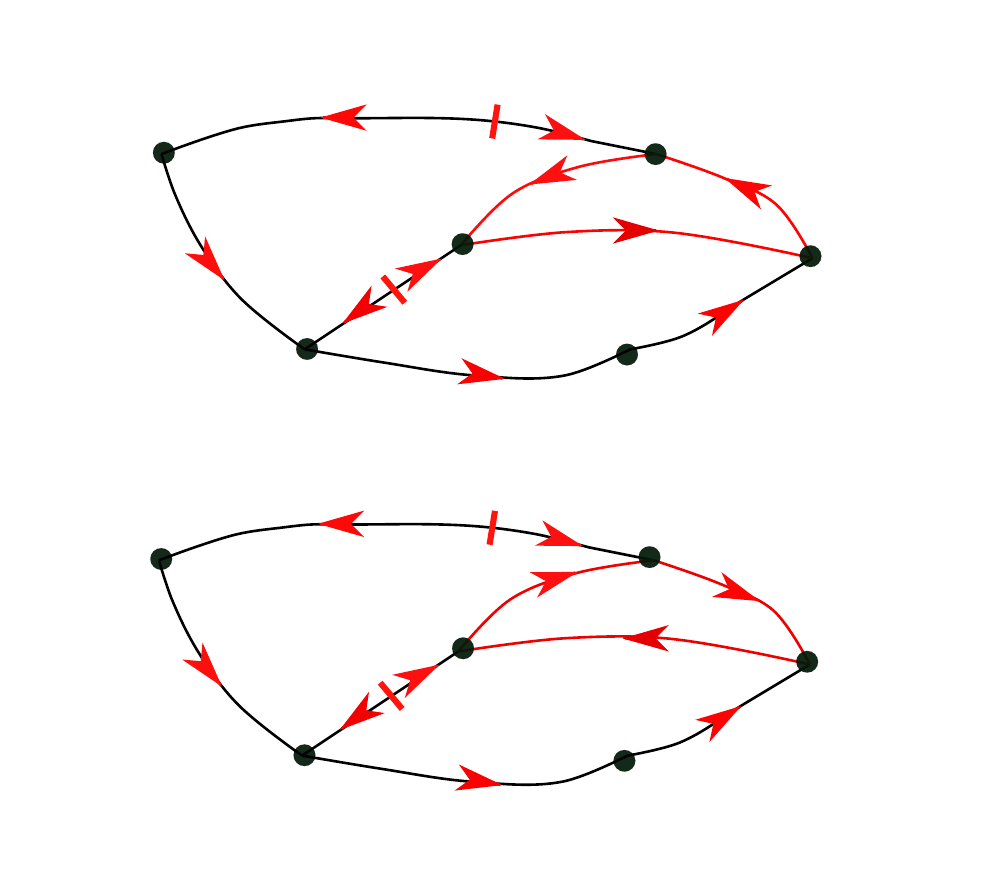}
	
	\caption{Two uni-cyclic spanning metric oriented subgraphs $\mathcal L$ and $\upbar{ \mathcal L}$ of the metric graph in the top picture of Figure \ref{taupiuomeno}. The orientations are red colored, the unique cycle is also red colored and oppositely oriented in the two pictures. Note that the two pictures are obtained by the middle and bottom pictures of Figure \ref{taupiuomeno} simply removing the x cut.}\label{unicicli}
	
\end{figure}
\smallskip

We can now compute $\partial_x\left(\sigma^2_\mathbf e(x)m_\mathbf e(x)\right)$ and obtain that this derivative is given by 3 different contributions. The first one is given by
\begin{equation}\label{primo}
2 b_{\mathbf e}(x)m_\mathbf e(x)\,,
\end{equation}
that is obtained summing terms of the form \eqref{dacit} that appear in the derivatives both when $\mathbf e\cap \mathcal C(T)=\emptyset$ as well as when
$\mathbf e\cap \mathcal C(T)\neq\emptyset$.

The remaining contributions are one positive and one negative and they arise from terms of the type \eqref{daspiegare}. The positive one is given by
\begin{equation}\label{second}
\sum_{L\in \mathbb L{\mathbf e}}\int_{\tilde{\mathcal{C}}(L)}dy_1,\dots dy_{k-1}e^{\int_{\mathcal L(\tilde{y}, \theta^*)}s}\prod_{v\in V}\mathcal W_v\big(\mathcal L(\tilde{y}, \theta^*)\big)\,.
\end{equation}
In the above formula $\theta^*$ denotes the orientation of the unique cycle
opposite to the natural one of $\mathbf e$ and the $\mathcal W_{.}$ terms give the weights to the nodes depending on the local orientation of the edges and they are defined like in \eqref{cong}. The negative term
is given by
\begin{equation}\label{terzo}
- \sum_{L\in \mathbb L{\mathbf e}}\int_{\tilde{\mathcal{C}}(L)}dy_1,\dots dy_{k-1}e^{\int_{\mathcal L(\tilde{y}, -\theta^*)}s}\prod_{v\in V}\mathcal W_v(\mathcal L(\tilde{y}, -\theta^*))\,,
\end{equation}
where $-\theta^*$ corresponds to the opposite orientation of the cycle with respect to $\theta^*$ and agrees therefore with the natural one of $\mathbf e$.

We can now prove the properties of the measure $m$.

\smallskip

\noindent $(1)$: Recall that the current probability along each edge is given by
\begin{equation}\label{tempera}
J_{\mathbf e}[m]:=-\frac 12\partial_x\left(\sigma_{\mathbf e}^2(x)m_\mathbf e(x)\right)+b_{\mathbf e}(x)m_\mathbf e(x)\,.
\end{equation}
If we insert the three terms obtained by the computation of $\partial_x\left(\sigma^2_\mathbf e(x)m_\mathbf e(x)\right)$ we obtain the following. The term \eqref{primo} matches exactly the last term in \eqref{tempera} so that the probability current
is exactly equal to the sum of the two terms in \eqref{second} and \eqref{terzo}. Since these terms do not depend on the specific point $x\in \mathbf e$ we have that the current \eqref{tempera} does not depend on $x$ too.

\smallskip

\noindent $(2)$: The divergence free condition is obtained by the following argument. We have that the non zero contributions to the current $J_{.}[m]$ on the metric graph are coming from the constant  (along each edge) terms \eqref{second} and \eqref{terzo}.
Given an $L\in \mathbb L$ we have a contribution to $J_{.}[m]$ that is given by a constant to be added to all the edges $\mathbf e$ that are in correspondence with edges of the unique cycle in $L$. This is a divergence free contribution since on each node the current entering is equal to the current exiting. Since a finite superposition of divergence free currents is divergence free we deduce the statement.

\smallskip

\noindent $(3)$ upper equations: In the computation of $m_{\mathbf{e}}(v)$ we can ignore, when considering the limit $y \to v$ in \eqref{vieris}, the contribution from integrations coming from cuts in a neighborhood of the node $v$.  More precisely consider $\mathbf e, \mathbf e'\in A(v)$ with $v\in V$ and $z\in \mathbf e$, $z'\in \mathbf e'$ such that $z,z'$ are at distance (along the edges of the metric graph) $\epsilon$ from $v$. We can write
\begin{equation}\label{verbali}
m_\mathbf e(z)=\left[\sum_{T\in \mathbb T}\int_{\mathcal {\upbar C}^\epsilon(T)}dy_1\dots dy_k R\left(\tau_z\left[\underline y\right]\right)+
\sum_{T\in \mathbb T}\int_{\mathcal {C}^\epsilon(T)}dy_1\dots dy_k R\left(\tau_z\left[\underline y\right]\right)\right],
\end{equation}
where $\mathcal{C}^\epsilon(T)$ is the collection of cutting points such that there is at least one cut at distance less or equal to $\epsilon$ from $v$ while $\mathcal{\upbar C}^\epsilon(T)$ is the complementary set, i.e. the collection of cutting points that are all at distance greater that $\epsilon$ from $v$. A formula similar to \eqref{verbali} can be written also for the point $z'$. In the limit $\epsilon\to 0$ the second term in \eqref{verbali} is negligible.

The basic observation is the following. Recall that by definition any metric arborescence has all the edges apart one oriented entering into any vertex. Consider a tree $T$ and some cutting points $\underline y\in \mathcal{\upbar C}^\epsilon(T)$. In this case all the edges $\mathbf e''\in A(v)$ such that $\mathbf e''\neq \mathbf e$ will be oriented locally around $v$ in $\tau_z[\underline y]$ entering into $v$ while instead $\mathbf e$ will be oriented exiting from $v$. In $\tau_{z'}[\underline y]$ we have instead that $\mathbf e'$ is oriented exiting from $v$ while instead all the other edges in $A(v)$ will be oriented entering into $v$. In particular in  $\tau_{z'}[\underline y]$ and $\tau_z[\underline y]$ all the branches of the two arborescenses will have the same orientation apart the two segments $(v,z)\subseteq \mathbf e$ and $(v,z')\subseteq \mathbf e'$ that are of size $\epsilon$.

We have therefore that for any $T$ and for any $\underline y\in \mathcal{\upbar C}^\epsilon(T)$
\begin{equation}\label{enfin}
\frac{R(\tau_z(\underline y))}{R(\tau_{z'}(\underline y))}=e^{2\int_x^zs}e^{2\int_{z'}^xs}\frac{W^+_x(\mathbf e)\sigma^2_{\mathbf e'}(z')}{W^+_x(\mathbf e')\sigma^2_{\mathbf e}(z)}\,.
\end{equation}
Taking the limit $\epsilon \to 0$ in \eqref{enfin} and using \eqref{cambiata} we finish the proof.
\end{proof}

Using the above result we can finally conclude. Let us call
\begin{equation}\label{insic}
\tilde\lambda_v:=\lim_{y\in \mathbf e, y\to v}\frac{\sigma^2_\mathbf e(y)m_\mathbf e(y)}{\alpha_{v,\mathbf e}}\,\qquad v\in V,\mathbf e\in A(v)\,.
\end{equation}
We observe by Theorem \ref{ilteo}  that $\tilde\lambda_v$ does not depend on the edge $\mathbf e\in A(v)$. Let us also introduce the probability measure $\mu$ as in \eqref{cms}
defined by
\begin{equation}\label{lamisura}
\left\{
\begin{array}{ll}
\mu_v:=\frac{\tilde\lambda_v\alpha_v}{Z} & v\in V\\
\mu_\mathbf e(x):=\frac{m_\mathbf e(x)}{Z} & x\in \mathbf e
\end{array}
\right.
\end{equation}
where $Z$ is the normalization constant determined by \eqref{norm}.
\begin{theorem}
The probability measure \eqref{lamisura} is the invariant measure of the process with generator satisfying \eqref{def}, provided that \eqref{cambiata} holds.
\end{theorem}
\begin{proof}
Conditions $(1)$ and $(2)$ in Lemma \ref{dot} depend just on the form of the measure $\mu$ on the edges and they are preserved by a multiplication of the measure by a constant factor. Since by Theorem \ref{ilteo} the measure $m$ satisfies these conditions and since the density of $\mu$ on each edge is obtained multiplying the density of $m$ by the factor $Z^{-1}$ we deduce that $\mu$ satisfies conditions $(1)$ and $(2)$.

Condition $(3)$ on Lemma \ref{dot} is satisfied by the fact that the limit on the right hand side of \eqref{insic} does not depend on $\mathbf e\in A(v)$ and by definition \eqref{lamisura}.

Since conditions $(1)$, $(2)$ and $(3)$ characterize the unique invariant measure we deduce that $\mu$ is the invariant measure.
\end{proof}

This construction of the invariant measure gives as special cases formula \eqref{raptorus} for a one dimensional ring, and formula \eqref{facile} in the case of a single interval since in that case the cut space is empty and there are no integrations to be done.

\smallskip

\subsection{The reversible case:} In the reversible case all the structure simplifies and the invariant measure can be simply computed in terms of an auxiliary finite state Markov chain on the set $V$. We hence start defining the auxiliary Markov chain. Given $v,w\in V$ we define the rate of jump from $v$ to $w$ across the edge $\mathbf e\in A(v)\cap A(w)$ as
\begin{equation}\label{befana}
q(v,w):=\left\{
\begin{array}{ll}
\alpha_{v,\mathbf e}e^{\int_0^{\ell_\mathbf e}s_\mathbf e} & \textrm{if}\ \mathbf e\in A^+(v)\cap A^-(w)\,,\\
\alpha_{v,\mathbf e}e^{-\int_0^{\ell_\mathbf e}s_\mathbf e} & \textrm{if}\ \mathbf e\in A^-(v)\cap A^+(w)\,.
\end{array}
\right.
\end{equation}
For simplicity we restrict to the case $|A(v)\cap A(w)|\leq 1$. The general case can be discussed very similarly.

We have the following characterization of reversible processes and their invariant measures.
\begin{theorem}\label{finiramai}
Consider a diffusion process on a metric graph $\mathcal{G}$ having generator $\mathcal A$ as in Definition \ref{def}. Then the process is reversible if and only if the Markov chain on $V$  with rates \eqref{befana} is reversible. In this case the invariant measure of the process is given by
\begin{equation}\label{lainva}
\mu=\left\{
\begin{array}{ll}
\mu_\mathbf e(x)=\frac{2c\pi_vq(v,w)}{\sigma^2_\mathbf e(x)}e^{\int_0^xs_\mathbf e(y)dy -\int_x^{\ell_{\mathbf e}}s_\mathbf e(y)dy}\,, & x\in \mathbf e\in A(v)\cap A(w)\,,\\
\mu_v=c\pi_v\alpha_v\,, & v\in V\,,
\end{array}
\right.
\end{equation}
where $\left(\pi_v\right)_{v\in V}$ is the unique invariant measure of the Markov chain with rates \eqref{befana} and $c$ is a suitable normalization constant.
\end{theorem}
\begin{proof}
Suppose that the Markov chain with rates \eqref{befana} is reversible with invariant measure $\pi$. Then by the detailed balance relationship
\begin{equation}\label{detbal}
\pi_vq(v,w)=\pi_wq(w,v)
\end{equation}
the measure $\mu$ is well defined.
We have to show that the measure \eqref{lainva} satisfies the conditions $(1')$ and $(3)$ in Lemma \ref{dot}. Condition $(1')$ is satisfied since on each edge the measure is given by the upper formula in \eqref{lainva} that coincides, by \eqref{facile2} \eqref{facile3}, with the simple case \eqref{facile} that corresponds to zero probability current.

If we compute the quantities appearing in condition $(3)$, that is formula \eqref{arrivanoirussi},
we obtain
\begin{equation}
\left\{
\begin{array}{ll}
\frac{1}{2}\sigma^2_\mathbf e(v)\mu_{\mathbf e}(v)=c\pi_v\alpha_{v,\mathbf e} & v\in V, \mathbf e\in A(v)\,,\\
\mu_v=c\pi_v\alpha_v & v\in V\,,
\end{array}
\right.
\end{equation}
so that condition $(3)$ is satisfied with $\lambda_v=c\pi_v$.

Conversely suppose that the diffusion process on a metric graph having generator $\mathcal A$ as in Definition \ref{def} is reversible. Then, since the probability current across each edge must be zero, then the density of the invariant measure $\mu$ on each edge has to be of the form
\begin{equation}
\mu_\mathbf e(x)=\frac{C_\mathbf e}{\sigma^2_\mathbf e(x)}e^{\int_0^xs_\mathbf e(y)dy -\int_x^{\ell_{\mathbf e}}s_\mathbf e(y)dy}\,,
\end{equation}
for some suitable constants $C_\mathbf e$.
Condition $(3)$ for the stationarity implies that there exists some positive numbers $\left(\lambda_v\right)_{v\in V}$ for which we have
\begin{equation}\label{cess}
\left\{
\begin{array}{ll}
C_\mathbf e=\lambda_v\alpha_{v,\mathbf e}e^{\int_0^{\ell_\mathbf e}s_\mathbf e(y)dy}\,, & \textrm{if}\ \mathbf e\in A^{+}(v)\\
C_\mathbf e=\lambda_v\alpha_{v,\mathbf e}e^{-\int_0^{\ell_\mathbf e}s_\mathbf e(y)dy}\,, & \textrm{if}\ \mathbf e\in A^{-}(v)\,.
\end{array}
\right.
\end{equation}
Take $\mathbf e\in A(v)\cap A(w)$ and compute $C_\mathbf e$ using \eqref{cess}
using the two different formulas at $v$ and $w$. We get
$$
C_\mathbf e=\lambda_vq(v,w)=\lambda_wq(w,v)\,.
$$
Normalizing to one the $\lambda$'s we have that the second identity above coincides with the detailed balance for the Markov chain with rates $q$.
\end{proof}

\end{document}